\documentclass[sn-mathphys]{sn-jnl}

\jyear{2021}%

\theoremstyle{thmstyleone}%

\theoremstyle{thmstyletwo}%

\theoremstyle{thmstylethree}%

\newtheorem{thm}{Theorem}[section]
\newtheorem{prop}[thm]{Proposition}
\newtheorem{lem}[thm]{Lemma}
\newtheorem{defi}[thm]{Definition}
\newtheorem{cl}[thm]{Claim}
\newtheorem{cor}[thm]{Corollary}

\newcommand{\BB}{\mathbb{B}}
\newcommand{\U}{\mathbb{U}}

\newcommand{\RR}{\mathbb{R}}
\newcommand{\CC}{\mathbb{C}}

\newcommand{\rre}{\mathop{\mathrm{Re}\,}}

\def\be{\begin{equation}}
\def\ee{\end{equation}}
\newcommand{\bcl}{\begin{cl}}
\newcommand{\ecl}{\end{cl}}
\newcommand{\bpf}{\begin{proof}}
\newcommand{\epf}{\end{proof}}
\newcommand{\bcor}{\begin{cor}}
\newcommand{\ecor}{\end{cor}}

\begin{document}

\title[Article Title]{The Boundary Schwarz lemma for harmonic and pluriharmonic mappings and some
generalizations}


\author*[1]{\fnm{Miodrag} \sur{Mateljevi\'c}}\email{miodrag@matf.bg.ac.rs}

\author[2]{\fnm{Nikola} \sur{Mutavd\v{z}i\'c}}\email{nikolam@matf.bg.ac.rs}

\affil*[1]{\orgdiv{Department of Mathematics}, \orgname{Serbian Academy of Science and Arts}, \orgaddress{\street{Kneza Mihaila 35}, \city{Belgrade}, \postcode{11000}, \state{Serbia}, \country{Serbia}}}

\affil[2]{\orgname{Mathmatical Institute of Serbian Academy of Science and Arts}, \orgaddress{\street{Kneza Mihaila 36}, \city{Belgrade}, \postcode{11000}, \state{Serbia}, \country{Serbia}}}


\abstract{We improve and generalize the classical Schwarz lemmas for planar harmonic
mappings into the sharp forms for harmonic mappings
between finite dimensional Euclidean unit ball and the unit ball in
Hilbert space, and present some applications to sharp boundary Schwarz
type lemmas for holomorphic and in particular pluriharmonic map-
pings between the unit balls in Hilbert and Banach spaces. In the second part of this article, using Burget’s estimate we establish the sharp boundary Schwarz
type lemmas for harmonic mappings between finite dimensional balls.  Since
here   we do not suppose in general  that maps fix the origin    this is a
generalization of Theorem 2.5 in the  Kalaj's  paper
\cite{DKal2016HeinzSchwarz}. At the end of this section, we derived some
interesting conclusions considering hyperbolic-harmonic function in the
unit ball, which  shows that Hopf's lemma is not applicable for those
functions.}




\maketitle

\section{Introduction}

The classical  Schwarz lemma is a result in complex analysis about
holomorphic functions from the open unit disk to itself.
The lemma is less celebrated than deeper theorems, such as the Riemann
mapping theorem, which it helps to prove.
 Although   it is one of the simplest results showing the rigidity of
holomorphic functions
Schwarz lemma has been generalized in various directions and it has become
a crucial theme in many branches of research in Mathematics for more than
a hundred years to  the  present day. There is vast literature related
to the subject, but here we cite mainly recent papers;  for a more
complete list of references  see
\cite{RGschMM1,CHPRbanach,MM-JMAA,BiltenAcad}  and the references therein
for  more fundamental results.

We only briefly discuss recent results that have affected our work .  We
draw the reader's attention the the result in subsection \ref{ss_sv}  was
obtained before the result
in the next subsection.

\subsection{Schwarz lemma  and Hilbert spaces}

In  \cite{Li_MM}, the first author of this paper  in a joint paper  with Li
considered  pluriharmonic and harmonic mappings $f$
defined on the unit ball $\mathbb{B}^n$, $n\geq2$, differentiable at  a
point $a$  on the
boundary of $\mathbb{B}^n$, and $f(\mathbb{B})$ satisfies some convexity
hypothesis at $f(a)$.
For those mappings $f$, they obtained versions of boundary Schwarz lemma
and the sharp estimate of the eigenvalue related to its Jacobian at $a$.

After writing the final version of the manuscript \cite{Li_MM}
 Hamada turned  attention \footnote{in communication with M. Mateljev\' c}
to the arxiv  paper \cite{CHPRbanach}. In \cite{CHPRbanach}, the authors
generalize the classical Schwarz lemmas of planar harmonic mappings into
the sharp forms for Banach spaces, and present some applications to sharp boundary
Schwarz type lemmas for pluriharmonic mappings in Banach spaces. Recently,
Hamada and Kohr published paper \cite{HamadaKohr}, where authors discussed
rigidity theorems on the boundary for holomorphic mappings. They explained difference of the constants obtained for unit ball and polydisc
and also presented a generalization for other bounded symmetric regions
in $\mathbb{C}^n$ and balanced domains in complex Banach spaces.

 In this paper we get further results using approaches from those papers.

The following result was obtained by I. Graham, H. Hamada and G. Kohr in
[\cite{GHK}, Proposition 1.8]  stated here as:
\begin{thm}[\cite{GHK}]\label{GHKTheorem} Let $\mathbb{B}_j$  be the unit
ball of a complex Hilbert space $H_j$
for  $j = 1,2$, respectively.
Let  $f: \mathbb{B}_1 \rightarrow \mathbb{B}_2$ be a pluriharmonic mapping.
Assume that $f$  is of
class $C^1$ at some point  $z_0 \in \partial\mathbb{B}_1$ and  $f(z_0) =
w_0 \in
\partial\mathbb{B}_2$. Then
there exists a
constant $ \lambda \in \mathbb{R}$  such that $Df(z_0)^*w_0 = \lambda
z_0$. Moreover,
$$\lambda \geqslant\frac{1-\rre(\langle f(0),w_0\rangle)}{2} > 0 .$$
\end{thm}

In the Section \ref{SpluriHilbert} we will improve this estimate.
Next  S. Chen, H. Hamada, S. Ponnusamy, R.
Vijayakumar in \cite{CHPRbanach}       observed  that  using  [\cite{GHK},
Proposition 1.8]   and the arguments similar to those in the proof of
their  Theorem 3.3 \cite{CHPRbanach}
one can obtain a better estimate:

\begin{prop}\label{propCHPR}
$$\lambda \geq \max \{\frac{2}{\pi}- \vert f(0)\vert , \frac{1- \rre(\langle
f(0),w_0\rangle)}{2}\}.$$
\end{prop}

Note that  under condition $f(0)=0$, in Theorem 1.1 (ii) and (iii) in
\cite{Li_MM}, it is shown  that  $\lambda\geq 2/\pi$. (But it  also
follows from the above Proposition \ref{propCHPR}).

Next in  \cite{CHPRbanach} a version of the boundary Schwarz
lemma  for the complex Banach spaces was proved:

\begin{thm} [Theorem 3.3\cite{CHPRbanach}]
Suppose that $B_X$ and $B_Y$ are the unit balls of the complex Banach
spaces $X$ and $Y$, respectively, and $f: B_X\rightarrow B_Y$  is a
pluriharmonic mapping. In addition, let $f$ be differentiable at $b\in b
B_X$ with $\vert f(b)\vert _Y = 1$. Then we have
\be
\vert Df(b)b\vert _Y\geq \max \{\frac{2}{\pi}- \vert f(0)\vert , \frac{1-\vert f(0)\vert }{2}\}.
\ee
\end{thm}

In Section \ref{SpluriBanach}  we proved  Theorem  \ref{NMbanach}  which
yields better estimate.
We leave to the interesting reader  to check this claim.
In Theorem \ref{DKalHil}, we establish Schwarz lemma on the boundary for
harmonic functions, mapping  the unit ball in $\RR^n$ into  unit ball in
some Hilbert space, which maps origin to origin. This is a generalization
of the work in paper \cite{DKal2016HeinzSchwarz}.

\subsection{Schwarz lemma for harmonic functions in several
variables}\label{ss_sv}

For a short discussion about  Schwarz lemma for  harmonic functions in the
planar case  see  Section \ref{s_App}.
As far as we  know  the study  related to  Schwarz lemma for real valued
harmonic functions, defined on the unit ball in $\RR^n$  with codomain
$(-1,1)$ was initiated by  Khavinson,  Burget,  Axler at al., for more
details see for example  \cite{BiltenAcad}.
Generalizations of Schwarz lemma for functions of several variables were  developed in the work of Burget \cite{burgeth1992} (see also the papers by H.A. Schwarz and E.J.P.G. Schmidt cited there), which were based on the integration of Poisson kernels over the so called polar caps, using the
spherical coordinates\footnote{we refer to this method as Burget's
spherical cap method}
and we used some formulas from that paper, which are described in the
first part of Section \ref{hyp-har}.
Khavinson  \cite{Khavinson}, using also spherical caps, indicates an
elementary argument that allows one to obtain sharp estimates of derivatives of bounded harmonic functions in the unit
ball in $\mathbb{R}^n$ (explicitly stated for $n=3$); this 3-dimensional
result has a physical interpretation.
It is worth mentioning that a similar idea occurs in the book
\cite{ABR1992} for maps which fix the origin in which case the spherical
cap is reduced to a hemisphere. Note  that researches have often overlooked   Burget's work (for more details  see  Section \ref{s_App}).

D. Kalaj \cite{DKal2016HeinzSchwarz}  considered
Heinz-Schwarz inequalities for harmonic mappings in the unit ball, which
is a version of Schwarz lemma on the boundary.

Recently, these ideas were discussed at the Belgrade Analysis Seminar, and
several recent results in this subject were obtained by the first author
and some of his associates:
M.Svetlik, A. Khalfallah, M. Mhamdi, B. Purti\'c, H.P. Li  and the second
author of this paper, see (\cite{KaMaMedJ}, \cite{MaSvAADM},
\cite{MMKhal}). For more details see the introduction of
paper \cite{BiltenAcad} by the first author of this paper.

In particular  we will use   here  [Proposition 4.3 \cite{Li_MM}]  which
is a corollary  of the estimate obtained  in  \cite{MaSvAADM}  (cf. also
\cite{MMKhal}),  stated here  as  Proposition \ref{pr_g2}.

In Section \ref{hyp-har} using  Burget's estimate  we establish Theorem
\ref{harmonic_bugheth} for harmonic mappings between finite dimensional
unit balls.
Since here   we do not suppose that maps fix the origin
this is a generalization of Theorem 2.5 in the  mentioned Kalaj's  paper.

At the end of this section, we derived some interesting conclusions
considering hyperbolic-harmonic function in the unit ball, which  shows
that Hopf's lemma is not applicable for those functions.

Chinese mathematicians have made a great contribution to this field but here we will mention only results that are related to the results
presented here\footnote{Z. Chen, Y. Liu and Y. Pan; S. Dai, H. Chen and Y.
Pan;X. Tang, T. Liu and W. Zhang;J.F. Zhu, etc}.

\section{Boundary Schwarz lemma for pluriharmonic mappings in Hilbert spaces}\label{SpluriHilbert}
Let $H$ be a complex Hilbert space with inner product $\langle\cdot,\cdot\rangle$.
Then $H$ ca be regarded as a real Hilbert space with inner product
$\rre\langle\cdot,\cdot\rangle$. Let $\vert \cdot\vert $ be the induced norm in $H$.
Let $\mathbb{B}$ be the unit ball of $H$. For each $z_0\in\partial\mathbb{B}$, the
tangent space $T_{z_0}(\partial\mathbb{B})$ is defined by
$$T_{z_0}(\partial\mathbb{B})=\{\beta\in H: \rre\langle z_0,\beta\rangle=0\}.$$

Let $H_1$ and $H_2$ be complex Hilbert spaces and let $\Omega$ be a domain in $H_1$.

\begin{defi}
A mapping $f:\Omega\rightarrow H_2$ is said to be differentiable at $z\in\Omega$ if
there exists a bounded linear map $Df(z)\in L_{\mathbb{R}}(H_1,H_2)$ such that
$$f(z+h)=f(z)+Df(z)h+o(\vert  h\vert  ),\mbox{ as } h\rightarrow 0.$$
\end{defi}
If $f$ is differentiable at each point of $\Omega$, then $f$ is said to be
differentiable on $\Omega$. In this case, the mapping
$$\mathcal{D}f:\Omega\rightarrow L_{\mathbb{R}}(H_1,H_2),\ z\mapsto Df(z)$$
is called the derivative (or differential) of $f$ on $\Omega$. If $\mathcal{D}f$ is
continuous in a neighborhood of $z$, the mapping $f$ is said to be of class $C^1$ at
$z$.
If $Df(z)$ is bounded complex linear for each $z\in\Omega$, then $f$ is said to be
holomorphic on $\Omega$.

\begin{defi}
A $C^2-$mapping $f:\mathbb{B}_1\rightarrow H_2$ is said to be pluriharmonic if the
restriction of the complex valued function $f_w(z)=\langle f(z),w\rangle$ to every
complex line is harmonic for each $w\in H_2$.
\end{defi}






Let $\mathbb{B}_j$ be the unit ball of a complex Hilbert space $H_j$ for  $j =
1,2,$  respectively.
Note that if $f$  is  differentiable  at  $z_0 \in \partial\mathbb{B}_1$  with values
in $H_2$,
then the adjoint operator
$Df(z_0)^*$ is defined by
$$\rre (\langle Df(z_0)^*w,z\rangle_{H_1} ) = \rre (\langle
w,Df(z_0)z\rangle_{H_2})\quad \mbox{for}\quad z \in H_1,\ w
\in H_2,$$

where $\langle\cdot,\cdot\rangle_{H_j}$  is the inner product of $H_j,\ j = 1,2$.

For $a\in H_1$ and $v\in T_a (H_1)$ (the tangent space at the point $a$), we define
the half space  $H(a,v)= \{y\in H_1: \rre\langle y-a,v\rangle< 0\}$. If $v$ is a unit
vector, then we use notation $n_a$ for it and write simply $H_a$ instead of
$H(a,v)$.
We also notice that in our  approach the following simple result is useful:
\bcl\label{531-1} Assume that $f$  is differentiable at a point $a\in H_1$ and let
$b=f(a)\in H_2$.   Then by the definition of adjoint operator, we have
$$ \rre\langle Df(a)Z, n_b \rangle=\rre \langle Z, Df(a)^* n_b \rangle,$$
for any $Z\in T_a (H_1)$. The following statements hold:

\begin{itemize}
\item[(i)] If $Df(a)$  maps $H_a$ into $H_b$,  then  $Df(a)^*n_b= \lambda n_a$,
where  $\lambda>0$.

\item[(ii)]
If further, $f$  maps $H_a$ into $H_b$,  then  $Df(a)^*n_b= \lambda n_a$, where
$\lambda\geq 0$.
In particular if     $Df(a)^* n_b\neq 0$, then  $\lambda>0$.

\item[(iii)]
In both cases  (i) and (ii),  we have
$\lambda=\vert Df(a)^* n_b\vert =\rre\langle Df(a)n_a, n_b\rangle$, and  $\lambda\leq
\vert Df(a)n_a\vert $.

\item[(iv)]
Let $\vert a\vert =1$. Define $u(x)=\rre\langle f(x),n_b\rangle$. Then $\lambda=D_r u(a)$.
\end{itemize}
\ecl
\bpf Here it is convenient to identify   $H_a$ and  $H_b$  with subsets  of  $T_a
(H_1)$ and  $T_b (H_2)$, respectively.
By hypothesis  $Df(a)$   maps $H_a$  into  $H_b$, and therefore  we have
$$0=\rre\langle Df(a)X, n_b\rangle=\rre\langle X, Df(a)^* n_b\rangle$$  for all  $X
\in T_a(H_a)$.
This shows that $X_0 =Df(a)^T n_b$ is orthogonal on $T_a(H_a)$.
In our setting, it   means  that   it  equals to $\lambda n_b$.  Then  by definition
 of the transpose, one has
$$\rre\langle Df(a)n_a, n_b\rangle=\rre\langle n_a, Df(a)^* n_b\rangle=\rre\langle n_a,
\lambda n_a\rangle= \lambda.$$
Since  $n_a \in  H_a$,  $Df(a)n_a \in  H_b$, by the definition of $H_b$, we first
conclude that  $\langle Df(a)n_a, n_b\rangle >0$,  and hence,  $\lambda >0$.
This completes the proof of (i).
For the proof of (ii), which  is similar to (i), we leave it to the interested
reader by considering two cases: $X_0=0$ and $X_0 \neq 0$.

(iii) is an immediate corollary  of (i) and (ii). (iv) is consequence of the fact that $D_r u(a)=Re\langle
Df(a)a,n_b\rangle=\lambda$
\epf

The proof of  Proposition \ref{4.3} and Theorem  \ref{NMbanach} is based on the following result:
\begin{prop}[Proposition 4.3 \cite{Li_MM}]\label{pr_g2}
Let $u:\mathbb{D}\rightarrow \mathbb{D}$ be a harmonic function such that
$u(0)=b$.
Assume that $u$ has a continuously extension to the boundary point $z_0\in \mathbb{T}$,  $u(z_0) = c\in\mathbb{T}$   and $a=\tan\frac{\vert P_c b\vert \pi}{4}$, where $P_c b=\langle b,c\rangle c$.
If $u$ is differentiable at $z_0$, then  $\vert D_r u(z_0)\vert \geq  \frac{2}{\pi} \frac{1-\vert a\vert }{1+\vert a\vert }$.
\end{prop}
Note here that $a=\vert a\vert $ and   $s^- (\vert P_c b\vert )=\frac{1-\vert a\vert }{1+\vert a\vert }=e^-(\vert a\vert )$.
\bpf Consider   the function $v=\mbox{Re} (\overline{c} u)$.
Then function $v$ is real-valued harmonic  and apply real version.
\epf

\begin{prop}\label{4.3}
Let $\mathbb{B}_j$  be the unit ball of a complex Hilbert space $H_j$
for  $j = 1,2$,
respectively. Let  $f: \mathbb{B}_1 \rightarrow \mathbb{B}_2$ be a pluriharmonic
mapping.
Assume that $f$  is differentiable at some point  $z_0 \in \partial\mathbb{B}_1$ and
 $f(z_0) = w_0 \in \partial\mathbb{B}_2$. Then
there exists a
constant $ \lambda \in \mathbb{R}$  such that $Df(z_0)^*w_0 = \lambda z_0$. Moreover,
$$\lambda \geqslant s^-(b)=\frac{2}{\pi}\cot(\frac{\pi}{4}(1+b)) > 0,\mbox{ where }
b=\rre(\langle f(0),w_0\rangle).$$

We note that $s^-(x)\geqslant \frac{1-x}{2},\  x\in (-1,1)$
\end{prop}

\begin{proof}
Let us consider function $u:\mathbb{U}\rightarrow (-1,1)$, defines with
$u(z)=\rre\langle f(zz_0),w_0\rangle$. Function $u$ will be harmonic function and we
have $u(0)=b$. Function $u$ has continuous extension an point $z_0\in\mathbb{T}$ and
we can check that $u(1)=1$. Applying Proposition \ref{pr_g2}, we get
$\vert D_r u(1)\vert \geqslant s^-(b)$. Also, we have that $D_r u(1)=Re\langle
Df(z_0)z_0,w_0\rangle=\lambda$.
\end{proof}

Suppose that function $f:\Omega\rightarrow H_2$, $\Omega$ is a domain in $H_1$ is
holomorphic in $\Omega$ and $z_0\in\Omega$ be any point. We define hermitian adjoint
operator $Df(z_0)^{\dagger}$ in the next manner

$$ \langle Df(z_0)^{\dagger}w,z\rangle_{H_1}  =  \langle
w,Df(z_0)z\rangle_{H_2}\quad \mbox{for}\quad z \in H_1,\ w
\in H_2,$$

where $\langle \cdot,\cdot\rangle_{H_j}$  is the inner product of $H_j,\ j = 1,2$.

\begin{lem}$($\cite{ChenLiuPan}$)$Let $\varphi_{\xi}(z)=A\frac{\xi-z}{1-\langle z,\xi\rangle}$ be the
holomorphic
automorphism of $\mathbb{B}_1$ where
$A:H_1\rightarrow H_1$ in the sense that $A(v)=s_{\xi}v+\frac{\xi\langle v,\xi\rangle}{1+s_\xi}$,
$s_\xi=\sqrt{1-\vert \xi\vert ^2}$ and $v\in H_1$. Then $\varphi_\xi$ is biholomorphic in a neighborhood of
$\overline{\mathbb{B}}_1$, and
$$A^2=s_\xi^2Id+\xi\langle\cdot,\xi\rangle,\ \ \ A\xi=\xi, \ \ \
\varphi_\xi^{-1}=\varphi_\xi, \ \ \ D\varphi_\xi(z)=A\left[-\frac{Id}{1-\langle
z,\xi\rangle}+\frac{(\xi-z)\langle \cdot,\xi\rangle}{(1-\langle
z,\xi\rangle)^2}\right].$$
\end{lem}
If we denote $P(v)=\xi\langle v,\xi\rangle$ it can be checked that $P^{\dagger}=P$.
From this $A^{\dagger}=A$ follows. Also, if $Q(v)=z\langle v,\xi\rangle$ and
$R(v)=\xi\langle v,z\rangle$ then $Q^{\dagger}=R$. Now, we have

$$D\varphi_{\xi}(z)^{\dagger}=\left[-\frac{Id}{1-\overline{\langle
z,\xi\rangle}}+\frac{\xi\langle \cdot,\xi-z\rangle}{(1-\overline{\langle
z,\xi\rangle})^2}\right]A$$
Let us denote $L_z=\left[-\frac{Id}{1-\overline{\langle
z,\xi\rangle}}+\frac{\xi\langle \cdot,\xi-z\rangle}{(1-\overline{\langle
z,\xi\rangle})^2}\right]$.

\begin{lem}\label{LemHil}$($\cite{ChenLiuPan}$)$For every $z_0\in \BB_1$ we have $D\varphi_{\xi}(z_0)^{\dagger}\varphi_{\xi}(z_0)=\frac{1-\vert \xi\vert ^2}{\vert 1-\langle
z_0,\xi\rangle\vert ^2}z_0$.
\end{lem}
\begin{proof}
By direct computation
$D\varphi_{\xi}(z_0)^{\dagger}\varphi_{\xi}(z_0)=L_{z_0}A^2\frac{\xi-z_0}{1-\langle
z_0,\xi\rangle}$. We can easily check that $A^2\frac{\xi-z_0}{1-\langle
z_0,\xi\rangle}=\xi-\frac{s^2 z_0}{1-\langle z_0,\xi\rangle}$. According to this, we
have
\begin{equation*}
\begin{split}
D\varphi_{\xi}(z_0)^{\dagger}\varphi_{\xi}(z_0) & = L_{z_0}(\xi-\frac{s^2
z_0}{1-\langle z_0,\xi\rangle}) = \\
& =-\frac{\xi}{1-\overline{\langle z_0,\xi\rangle}}+\frac{\xi\langle
\xi,\xi-z_0\rangle}{(1-\overline{\langle
z_0,\xi\rangle})^2}+\frac{s^2z_0}{\vert 1-\langle z_0,\xi\rangle\vert ^2}-\frac{s^2\xi\langle
z_0,\xi-z_0\rangle}{\vert 1-\langle z_0,\xi\rangle\vert ^2(1-\overline{\langle
z_0,\xi\rangle})}= \\
& = -\frac{\xi(1-\langle \xi,z_0\rangle)}{(1-\overline{\langle
z_0,\xi\rangle})^2}+\frac{\xi\langle \xi,\xi-z_0\rangle}{(1-\overline{\langle
z_0,\xi\rangle})^2}+\frac{s^2z_0}{\vert 1-\langle z_0,\xi\rangle\vert ^2}-\frac{s^2\xi(\langle
z_0,\xi\rangle-1)}{\vert 1-\langle z_0,\xi\rangle\vert ^2(1-\overline{\langle
z_0,\xi\rangle})} = \\
& = \frac{-\xi(1-\vert \xi\vert ^2)}{(1-\overline{\langle
z_0,\xi\rangle})^2}+\frac{s^2z_0}{\vert 1-\langle z_0,\xi\rangle\vert ^2}-\frac{s^2\xi(\langle
z_0,\xi\rangle-1)}{\vert 1-\langle z_0,\xi\rangle\vert ^2(1-\overline{\langle
z_0,\xi\rangle})} = \\
& = \frac{-\xi s^2(1-\langle z_0,\xi\rangle))}{(1-\overline{\langle
z_0,\xi\rangle})^2(1-\langle z,\xi\rangle)}+\frac{s^2z_0}{\vert 1-\langle
z_0,\xi\rangle\vert ^2}-\frac{s^2\xi(\langle z_0,\xi\rangle-1)}{\vert 1-\langle
z_0,\xi\rangle\vert ^2(1-\overline{\langle z_0,\xi\rangle})}=\\
& = \frac{1-\vert \xi\vert ^2}{\vert 1-\langle z_0,\xi\rangle\vert ^2}z_0.
\end{split}
\end{equation*}
\end{proof}

Let $V_1$ and $V_2$ be two complex vector space. We define sets or real linear, complex linear and complex antilinear operators between $V_1$ and $V_2$ in the following sense.

If $L:V_1\rightarrow V_2$ is additive linear operator, then
\begin{align*}
    L\in\mathcal{L}_{\RR}(V_1,V_2)\iff \forall\lambda\in\RR,\zeta\in V_1 : L(\lambda \zeta)=\lambda L(\zeta) \\
    L\in\mathcal{L}_{\CC}(V_1,V_2)\iff \forall z\in\CC,\zeta\in V_1 : L(z \zeta)=z L(\zeta) \\
    L\in\overline{\mathcal{L}}_{\CC}(V_1,V_2)\iff \forall z\in\CC,\zeta\in V_1 : L(z\zeta)=\bar{z} L(\zeta)
\end{align*}
It can be shown that the next statement holds: $\mathcal{L}_{\RR}(V_1,V_2)=\mathcal{L}_{\CC}(V_1,V_2)\oplus\overline{\mathcal{L}}_{\CC}(V_1,V_2).$

First, we check that $\mathcal{L}_{\CC}(V_1,V_2)\cap\overline{\mathcal{L}}_{\CC}(V_1,V_2)=\{0\}$. If we argue by contradiction, we assume that there exists complex both linear and antilinear operator $L$ between $V_1$ and $V_2$ and $\zeta\in V_1$, such that $L(\zeta)\neq 0$. Then $L(i\zeta)=iL(\zeta)=-iL(\zeta)$ so $L(\zeta)=0$, which is contradiction.

Now, let us perceive arbitrary real linear operator $L$ from $V_1$ into $V_2$. We can define operators $L_1,L_2:V_1\rightarrow V_2$ such that $L_1(\zeta)=\frac{1}{2}(L(\zeta)-iL(i\zeta))$ and $L_2(\zeta)=\frac{1}{2}(L(\zeta)+iL(i\zeta))$. We argue that $L=L_1+L_2$ where $L_1\in\mathcal{L}_{\CC}(V_1,V_2),L_2\in\overline{\mathcal{L}}_{\CC}(V_1,V_2)$. If we regard $z=x+iy$ as any complex number, and $\zeta\in V_1$ arbitrary, then
\begin{align*}
  L_1(z\zeta) & = \frac{1}{2}(L((x+iy)\zeta)-iL(i(x+iy)\zeta))=\frac{1}{2}(L(x\zeta+yi\zeta)-iL(-y\zeta+xi\zeta))=\\
   & = \frac{1}{2}(xL(\zeta)+yL(i\zeta)+iyL(\zeta)-ixL(i\zeta))= \frac{1}{2}((x+iy)L(\zeta)-(x+iy)iL(i\zeta))=\\
   & = (x+iy)L_1(\zeta) = zL_1(\zeta).
\end{align*}
Analogously to this, we can get
\begin{align*}
  L_2(z\zeta) & = \frac{1}{2}(L((x+iy)\zeta)+iL(i(x+iy)\zeta))=\frac{1}{2}(L(x\zeta+yi\zeta)+iL(-y\zeta+xi\zeta))=\\
   & = \frac{1}{2}(xL(\zeta)+yL(i\zeta)-iyL(\zeta)+ixL(i\zeta))= \frac{1}{2}((x-iy)L(\zeta)+(x-iy)iL(i\zeta))=\\
   & = (x-iy)L_1(\zeta) = \bar{z}L_2(\zeta).
\end{align*}

\bcl Let $H_1$ and $H_2$ be two complex Hilbert spaces and $L:H_1\rightarrow H_2$ be a bounded complex linear operator. Then $L^*=L^\dag$.

\ecl
\begin{proof}
Now, assume that $L$ is bouned real linear operator from $H_1$ to $H_2$. Then, there are unique bounded operators $L_1$ and $L_2$, complex linear and complex antilinear, respectively, which satisfies $L=L_1+L_2$. For this operators, we can find bounded complex linear operator $L_1^{\dagger}$ such that $\langle L_1^{\dagger}(w),z\rangle=\langle w,L_1(z)\rangle$, and bounded, complex antilinear operator $L_2^{\ddagger}$ defined with expression $\langle L_2^{\ddagger}(w),z\rangle=\langle w,L_2(z)\rangle$, for all $z\in H_1, w\in H_2$. We argue that $L_1^*=L_1^{\dagger}$ and $L_2^*=L_2^{\ddagger}$. First, since both complex linear and complex antilinear operators are real linear, we can define real adjonit for this operators. Also, if  $\langle L_1^{\dagger}(w),z\rangle=\langle w,L_1(z)\rangle$ we get $\rre\langle L_1^{\dagger}(w),z\rangle=\rre\langle w,L_1(z)\rangle$, for all $z\in H_1, w\in H_2$. The same argument stands for operator $L_2^{\ddag}$.
\end{proof}

\begin{prop} \label{ProLi}
Let $\mathbb{B}_j$  be the unit ball of a complex Hilbert space $H_j$
for  $j = 1,2$,
respectively. Let  $f: \mathbb{B}_1 \rightarrow \mathbb{B}_2$ be a pluriharmonic
mapping such that $f(\xi)=0$.
Assume that $f$  is differentiable at some point  $z_0 \in \partial \mathbb{B}_1$
and  $f(z_0) = w_0 \in \partial \mathbb{B}_2$. Then
there exists a
constant $ \lambda \in \mathbb{R}$  such that
$$Df(z_0)^*w_0= \lambda \frac{1-\vert \xi\vert ^2}{\vert 1-\langle z_0,\xi\rangle\vert ^2}z_0,$$ where
$\lambda\geq\frac{2}{\pi}$.
\end{prop}

\begin{proof}
Let $\varphi_{\xi}(z)=A\frac{\xi-z}{1-\langle z,\xi\rangle}$ be the holomorphic
automorphism of $\mathbb{B}_1$ where
$A=s_{\xi}Id+\frac{\xi\langle \cdot,\xi\rangle}{1+s_\xi}$,
$s_\xi=\sqrt{1-\vert \xi\vert ^2}$.

Assume that $\varphi_\xi(z_0)=p\in\partial \mathbb{B}_2$.
Let $g(z)=f \circ\varphi_\xi(z)$.
Then $g$ is a pluriharmonic mapping of $\mathbb{B}_1$ into $\mathbb{B}_2$ satisfying
$$g(0)=f\circ \varphi_\xi(0)=f(\xi)=0,$$
and
$$g(p)=f\circ \varphi_\xi(p)=f(z_0)=w_0\in\partial \mathbb{B}_2.$$
According to Theorem \ref{4.3} we know that there exists a nonnegative
number $\lambda\in\mathbb{R}$
such that $$D^*_g(p)w_0=\lambda p.$$
From   $\varphi_\xi^2=Id$  it follows  that  $D{\varphi_\xi}(p)
D{\varphi_\xi}(x_0) =Id$  and therefore   (1)  $D{\varphi_\xi}(x_0)^*
D{\varphi_\xi}(p)^*=Id$.
Since
$Dg(p)= Df(x_0) D{\varphi_{\xi}}(p)$, we have    $Dg(p)^*=
(Df(z_0)D{\varphi_{\xi}}(p))^*= D{\varphi_{\xi}}(p)^*Df(z_0)^*$ and
therefore  (2) $D{\varphi_{\xi}}(p)^*D_f(z_0)^*w_0=\lambda p$.  By (1)
and (2)  we find  $Df(z_0)^*w_0= \lambda D{\varphi_\xi}(z_0)^*p$.

$$g(p)=f\circ \varphi_\xi(p)=f(z_0)=w_0\in\partial B_2.$$
According to Theorem \ref{4.3} we know that there exists a nonnegative
number $\lambda\in\mathbb{R}$
such that $$Dg(p)^*w_0=\lambda p.$$
From   $\varphi_\xi^2=Id$  it follows  that  $D_{\varphi_\xi}(p)
D_{\varphi_\xi}(z_0) =Id$  and therefore   (1)  $D\varphi_\xi(x_0)^*
D\varphi_\xi(p)^*=Id$.
Since
$Dg(p)= Df(x_0) D\varphi_{\xi}(p)$, we have    $Dg(p)^*=
(D_f(z_0)D\varphi_{\xi}(p))^*= D\varphi_{\xi}(p)^*Df(z_0)^*$ and
therefore  (2) $D\varphi_{\xi}(p)^*Df(z_0)^*w_0=\lambda p$.  By (1)
and (2)  we find that
$$Df(z_0)^*w_0= \lambda D\varphi_\xi(z_0)^*p.$$
From previous Lemma we concluded that $\langle z_0,D\varphi_{\xi}(z_0)^{\dagger}
p\rangle=\mu$, where $\mu=\frac{1-\vert \xi\vert ^2}{\vert 1-\langle z_0,\xi\rangle\vert ^2}$. Now, we
can conclude that $\langle D\varphi_{\xi}(z_0) z_0,p\rangle=\mu$, from which
$\rre\langle D\varphi_{\xi}(z_0) z_0,p\rangle=\mu$ follows. From Theorem 1.1 we
conclude that $D\varphi_\xi(z_0)^* p=\mu_1 a$, for some $\mu_1>0$. From the proof of
Theorem 1.1 we have $\mu_1=\rre\langle D\varphi_\xi(z_0)z_0,p\rangle=\langle
D\varphi_\xi(z_0)z_0,p\rangle=\mu$.






\end{proof}

Suppose that $\Omega\subset\RR^n$ to be a domain and $H$ is a Hilbert space. Let $f:\Omega\rightarrow H$ be a function such that $f\in C^2(\Omega)$. We define partial derivatives with respect to coordinates $x_i,i=1,\ldots n$ of the base $\{e_1,...,e_n\}$ in $\RR^n$ at the point $a\in\Omega$ with formula:
$$\frac{\partial f}{\partial x_i}(a)=Df(a)e_i$$.
\begin{defi}
Function $f$ is harmonic in domain $\Omega$ if $\sum\limits_{i=1}^n\frac{\partial^2 f}{\partial x_i^2}(a)=0$ for every $a\in\Omega$.
\end{defi}
Let us denote with $\BB$ an unit ball in the space $H$, and $\mathbf{B}^n$ unit ball in $\RR^n$.

It is well-known that a harmonic function $u\in L^{\infty}(\mathbf{B}^n)$ has the following integral representation
$$u(x)=\mathcal{P}[f](x)=\int_{\mathbf{S}^{n-1}}P(x, \zeta)f(\zeta)d\sigma(\zeta),$$
where $f$ is the boundary function of $\mathbb{S}^{n-1}$, and
$$P[x, \zeta]=\frac{1-\vert  x\vert  ^2}{\vert  x-\zeta\vert  ^n}\ \ \ ,\ \ \ \zeta\in\mathbf{S}^{n-1}$$ is the Poisson kernel and
$\sigma$ is the unique normalized rotation invariant Borel measure on $\mathbf{S}^{n-1}$.
According to \cite{DKal2016HeinzSchwarz}, we know that if $u$ is a harmonic self-mapping
of $\mathbf{B}^n$ such that $u(0)=0$, then
\begin{equation}\label{kalaj 1.4}
\vert  u(x)\vert  \leq U(rN),
\end{equation}
where $r=\vert  x\vert  $, $N=\{0, \cdots ,0, 1\}$ and $U$ is a harmonic function of $\mathbf{B}^n$ into
$[-1, 1]$ defined by
\begin{equation}\label{kalaj 1.5}
U(x)=P[\chi_{S^+}-\chi_{S^-}](x)
\end{equation}
where $\chi$ is the indicator function and $S^+=\{x\in\mathbf{S}^{n-1}:  x_n\geq 0\}$,
$S^-=\{x\in\mathbf{S}^{n-1}:  x_n\leq0\}$. We refer to \cite[Chapter 6]{ABR1992} for more details.

Recall that the {\it hypergeometric function} $_pF_q$ is defined for $\vert x\vert <1$ by the power series (\cite[(2.1.2)]{Landrews})
$$_pF_q[a_1,a_2,\dots,a_p;b_1,b_2,\ldots,b_q;x]=\sum_{n=0}^\infty\frac{(a_1)_n\cdots(a_p)_n}{(b_1)_n\cdots(b_q)_n}\frac{x^n}{n!}.$$
Here $(a)_n$ is the {\it Pochhammer symbol} and given as follows $(a)_n=\frac{\Gamma(n+a)}{\Gamma(a)}$.

The following result is the so-called {\it Heinz-Schwarz inequalities}.

\begin{lem}\label{Kalaj2016} \cite[Lemma 2.3]{DKal2016HeinzSchwarz}
The function $V(r)=\frac{\partial U(rN)}{\partial r}$, $0\leq r\leq1$ is decreasing on the interval $[0, 1]$,
and we have
\begin{equation}\label{Cn}
V(r)\geq V(1)=C_m=:\frac{m!\left(1+m-(m-2) _2F_1\left[\frac{1}{2},1;\frac{3+m}{2};-1\right]\right)}{2^{3m/2}\Gamma\left[\frac{1+m}{2}\right]\Gamma\left[\frac{3+m}{2}\right]}.
\end{equation}
\end{lem}

We refer the readers to \cite[Remark 2.7]{DKal2016HeinzSchwarz} for more details on the constant $C_m$ and related functions, when $m=2, 3, 4$.

A version of  Theorem  1.2  \cite{Li_MM} holds  for harmonic functions, where codomain is unit ball $\mathbb{B}$ in Hilbert space.

\begin{thm}\label{DKalHil}
Suppose that $f:\mathbf{B}^n\rightarrow\BB$ is a harmonic function, such that $f(0)=0$, and $f$ has continuous extension to the point $a\in\partial\mathbf{B}^n$ such that $f(a)=b\in\partial\BB$.
Then  $(1) \limsup\limits_{r\rightarrow 1_-}  \vert D_rf(ra)\vert \geq C_{n}$.

Suppose   in addition that  $f$ has  differentiable extension to  $a$.
\begin{itemize}
\item[(i)]
Then
there exists a positive number $\lambda\in\mathbb{R}$ such that ${Df(a)}^{*}b=\lambda a$ and
\item[(ii)]  $$\lambda  \geq C_{n},$$
where $C_{n}$ is given by $(\ref{Cn})$.
\item[(iii)] In particularly if $n=2$,  we have $\lambda\geq\frac{2}{\pi}$. This is sharp.
\end{itemize}
\end{thm}
\bpf (i) follows from Claim \ref{531-1}.
Set  $u= \rre\langle f,b\rangle$.
Since  $u$ is harmonic  and it  maps  $\mathbf{B}^{n}$ into $(-1,1)$,
$u(0)=0$ and     $u(a)=1$,
using
Theorem 6.16 \cite{ABR1992}  we have   $u(x) \leq U(rN)$  and therefore
$$1-u(x) \geq 1-U(rN),\ \ \ \mbox{for} \ \ \ r=\vert x\vert <1.$$
\ref{Kalaj2016}), and \cite[Lemma 2.4]{DKal2016HeinzSchwarz},we find
Hence
$$\frac{1-u(x)}{1-\vert x\vert }\geq \frac{1-U(rN)}{1-r}.$$

Next if we define  $u_0(t)= u(ta)$ and $U_0(t)= U(ta)$, $0<t<1$, for every
  $0<t<1$  there are  $c_t,d_t\in (t,1)$   such that
$1-u_0(t)=u_0'(c_t)(1-t)$,   $1-U_0(t)=U_0'(d_t)(1-t)$ and  $u_0'(c_t)\geq
U_0'(d_t)$.
Hence  by  Lemma \ref{Kalaj2016},  $u_0'(c_t) \geq  C_n$  and therefore
we get  (1).
If   in addition   $f$ has  differentiable extension to  $a$, then

$$D_ru(a)=\lim\limits_{\vert x\vert \rightarrow1^-}\frac{1-u(x)}{1-\vert x\vert }\geq\lim\limits_{r\rightarrow1^-}\frac{1-U(rN)}{1-r}=\left.\frac{\partial
U(rN)}{\partial r}\right\vert _{r=1}= C_{n}.$$
Since  by Claim \ref{531-1} (iv),  $\lambda=\rre\langle
D_rf(a),b\rangle=D_ru(a)$,  (ii) follows.  An application of Proposition
\ref{pr_g2} yields  (iii).
\epf

\section{Boundary Schwarz lemma and Banach spaces}\label{SpluriBanach}

We will use notation from \cite{CHPRbanach}.
Let $X$ and $Y$ be real or complex Banach spaces with norm $\vert \cdot \vert _X$ and $\vert \cdot
\vert _Y$ respectively. We denote with $\mathcal{L}(X,Y)$ the space of all continuous
linear operators from $X$ into $Y$ with the standard operator norm
$$\vert  A\vert  =\sup\limits_{x\in X\setminus\{0\}}\frac{\vert  Ax\vert  }{\vert  x\vert  },$$
where $A\in\mathcal{L}(X,Y)$. Then $\mathcal{L}(X,Y)$ is a Banach space with respect
to this norm. Denote by $X^*$ the dual space of the real or complex Banach space
$X$. For $x\in X\setminus \{0\}$, let
$$T(x)=\{l_x\in X^*:l_x(x)=\vert  x\vert  \mbox{ and }\vert  l_x\vert  =1\}.$$
Then the well known Han-Banach theorem implies that $T(x)\neq\varnothing.$

Let $f$ be a mapping of a domain $\Omega\subset X$ into real or complex Banach space
$Y$, where $X$ is a complex Banach space. We say that $f$ is differentiable at
$z\in\Omega$ if there exists a bounded real linear operator $Df(z):X\rightarrow Y$
such that
$$\lim\limits_{\vert  h\vert  \rightarrow 0^+}\frac{\vert  f(z+h)-f(z)-Df(z)h\vert  }{\vert  h\vert  }=0.$$
Here $Df(z)$ is called the Fr\'{e}chet derivative of $f$ at $z$. If $Y$ is a complex
Banach space and $Df(z)$ is bounded complex linear for each $z\in \Omega$, then $f$
is said to be holomorphic on $\Omega$.
Let $\Omega$ be a domain in a complex Banach space $X$. A mapping $f$ of $\Omega$
into a real or complex Banach space $Y$ is said to be pluriharmonic if the
restriction of $l\circ f$ to every holomorphic curve is harmonic for any $l\in
Y^*$.

\begin{thm}\label{NMbanach}
Suppose that $B_1$ and $B_2$ are the unit balls of the complex Banach spaces $X$ and
$Y$, respectively, and $f:B_1\rightarrow B_2$ is a pluriharmonic mapping. Assume
that the function $f$ is differentiable at $b\in\partial B_X$ with $\vert  f(b)\vert  =1.$
Then we have
$$\vert   Df(b)b\vert  \geq s^-(\vert  f(0)\vert  ).$$
\end{thm}

\begin{proof}
We consider the function $p(z)=\rre(l_{f(b)}(f(zb)))$, for $z\in\U$.
Since $f$ is pluriharmonic we have that the function $p$ is harmonic function on
$\U$. Also, from $\vert  l_{f(b)}\vert  =1$ we get $\vert \rre(l_{f(b)}(f(zb)))\vert \leq
\vert l_{f(b)}(f(zb))\vert \leq \vert  f(zb)\vert  < 1$, so we get that the function $p$ maps the unit
disc into interval $(-1,1).$ From definition of $l_{f(b)}$ we can conclude that
$p(1)=1$. Also, we have that $\vert p(0)\vert =\vert \rre l_{f(b)}(f(0))\vert \leq \vert l_{f(b)}(f(zb))\vert \leq
\vert  f(0)\vert  $. Now we can conclude that
$$\vert D_r p(1)\vert \geq s^-(\vert p(0)\vert )\geq s^-(\vert  f(0)\vert  ),$$
since the function $s^-$ is decreasing on $(-1,1)$.
On the other hand, we have that $\vert D_r p(1)\vert \leq\vert  Df(b)b\vert  $. Indeed, we have that
$$
    \vert D_r p(1)\vert =\lim\limits_{r\rightarrow
1^-}\frac{\vert p(1)-p(r)\vert }{1-r}=\lim\limits_{r\rightarrow 1^-}\left\vert \rre
l_{f(b)}\frac{f(b)-f(rb)}{1-r}\right\vert =\vert \rre l_{f(b)}Df(b)b\vert \leq\vert  Df(b)b\vert  ,
$$
which concludes our proof.
\end{proof}

\section{Hyperbolic harmonic functions in higher dimensions}\label{hyp-har}

We use notation from \cite{burgeth1992}. Let $\mathbf{B}^n$ be unit ball in $\RR^n$ and $\mathbf{S}^{n-1}$ be the surface of unit ball and $\Delta$ is Laplacian partial differential operator. Consider next Laplace-Beltrami operator
$$\Delta_0=\frac{1-\vert  x\vert  ^2}{4}(\Delta +\frac{2(n-2)}{1-\vert  x\vert  ^2}\langle x,\nabla\rangle).$$
Any twice continuously differentiable function $h$ which is defined on $\mathbf{B}^n$ and fulfills $\Delta_0 h=0$ is called \textit{hyperbolic harmonic} on $\mathbf{B}^n$.

In the sequel we will use some specific properties of both harmonic and hyperbolic-harmonic kernel, which are listed below:

\begin{itemize}

\item[a)] There exists a Poisson formula for hyperbolic harmonic as well as for harmonic functions on $\mathbf{B}^n$. Let $\sigma$ denote the usual surface measure on $\mathbf{S}^{n-1}$ and $f$ is $\sigma$-integrabile function on $\mathbf{S}^{n-1}$.  Set $x\in\mathbf{B}^{n}$ and $\eta\in\mathbf{S}^{n-1}$. Depending on whether we define $P(x,\eta)$ as
$$\frac{1}{\sigma(\mathbf{S}^{n-1})}\frac{1-\vert  x\vert  ^2}{\vert  x-\eta\vert  ^n}\mbox{ or as } \frac{1}{\sigma(\mathbf{S}^{n-1})}\frac{(1-\vert  x\vert  ^2)^{n-1}}{\vert  x-\eta\vert  ^{2(n-1)}},$$
we get a harmonic or a hyperbolic harmonic function on $\mathbf{B}^n$ by
$$h(x)=P[f](x)=\int_{\mathbf{S}^{n-1}}P(x,\eta)f(\eta)\mathrm{d}\sigma(\eta).$$
In the sequel we use the same notation $P$ for both Poisson kernels.

\item[b)]

\begin{equation}\label{kernel}
1=P[\mathbf{1}](x)=\int_{\mathbf{S}^{n-1}}P(x,\eta)\mathrm{d}\sigma(\eta).
\end{equation}

Where $\mathbf{1}(\eta)=1$ for all $\eta\in\mathbf{S}^{n-1}$ is a constant function.

This is immediate consequence of the fact that constant functions belongs to classes of hyperbolic and hyperbolic functions, both respectively.

\item[c)] Harmonic (resp. hyperbolic-harmonic) functions possess mean value property with respect to (hyperbolic) spheres.

\item[d)]The theorem of Fatou, concerning the $\sigma-$a.e, existence of non-tangential limits, is valid in both cases.
\end{itemize}
For convenience we set
\begin{align}\label{extremal}
M_c^n(\vert  x\vert  )=2P[\chi_{S(c,\tilde{x})}](x)-1\\
m_c^n(\vert  x\vert  )=2P[\chi_{S(c,-\tilde{x})}](x)-1,
\end{align}
where $x\in\mathbf{B}^n, \tilde{x}=\frac{x}{\vert  x\vert  }$ and $S(c,\tilde{x})$ denotes the polar cap with center $\tilde{x}$ and $\sigma-$measure $c$. Also, $\chi_A$ is an indicator function of the set $A$.
It is easy to verify that the expressions on the right hand side of (\ref{extremal}) inherit the rotational invariance of the measure $\sigma$.

For derivation on the explicit formula (\ref{ext1}), we refer to paper of first author \cite{MSS061021Fil}, specifically, to the Proposition 5.10. In this proposition we use the following notation: $\sigma_{n-1}$ is surface area of the sphere $\mathbf{S}^{n-1}$ and $\varphi$ is an angle between radius vector of point $\eta\in\mathbf{S}^{n-1}$ and radius vector of the point $\tilde{x}.$

\begin{prop}[\cite{MSS061021Fil}, Proposition 5.10]\label{MMThmSphere}
  If $f$ is a function on $\mathbf{S^{n-1}}$ depending only on $\varphi$, then
  $$\int_{\mathbf{S^{n-1}}}f(\eta)\mathrm{d}\sigma(\eta)=\sigma_{n-2}\int_{0}^{\pi}f(\varphi)\sin^{n-2}\varphi\mathrm{d}\varphi.$$
\end{prop}

Since $\vert  x-\eta\vert  ^2=1-2r\cos\varphi+r^2$ we have that both our kernels depend only on $\varphi$. Let us define $\sigma_*(n)=\frac{\sigma_{n-2}}{\sigma_{n-1}}$. Using formula $\sigma_{n-1}=\frac{2\pi^{n/2}}{\Gamma(\frac{n}{2})}$ we get $\sigma_*(n)=\frac{1}{\sqrt{\pi}}\frac{\Gamma(n/2)}{\Gamma((n-1)/2)}$.

Using Proposition \ref{MMThmSphere} we can rewrite (\ref{extremal}) as
\begin{align}\label{ext1}
  M_c^n(\vert  x\vert  )=2\sigma_*(n)(1-\vert  x\vert  ^2)^\nu\int_{0}^{\alpha(c)}\frac{\sin^{n-2}t}{(1-2\vert  x\vert  \cos t+\vert  x\vert  ^2)^{\mu}}\mathrm{d}t-1 \\
  m_c^n(\vert  x\vert  )=2\sigma_*(n)(1-\vert  x\vert  ^2)^\nu\int_{\pi-\alpha(c)}^{\pi}\frac{\sin^{n-2}t}{(1-2\vert  x\vert  \cos t+\vert  x\vert  ^2)^{\mu}}\mathrm{d}t-1,
\end{align}
where $(\nu,\mu)=(1,n/2)$ in harmonic case and $(\nu,\mu)=(n-1,n-1)$ in the hyperbolic-harmonic case and $\alpha(c)$ is the spherical angle of $S(c,\tilde{x})$.
\begin{thm}\cite{burgeth1992}
  Let $h$ be a harmonic or hyperbolic-harmonic function taking values in $(-1,1)$ and $h(0)=a, -1<a<1$. Then for $c=\frac{a+1}{2}$ and all $x\in\mathbf{B}^n$
  $$m_c^n(\vert  x\vert  )\leq h(x)\leq M_c^n(\vert  x\vert  ).$$

  Equality on the right (resp., left-) hand side for some $z\in \mathbf{B}^n\setminus\{0\}$ implies
  \begin{align*}
    h(x)=2P[\chi_{S(c,\tilde{z})}](x)-1, \\
    (\mbox{respectively, }h(x)=2P[\chi_{S(c,-\tilde{z})}](x)-1)
  \end{align*}
for all $x\in\mathbf{B}^n$
\end{thm}

\begin{lem}
Let $(\nu,\mu)=(1,n/2)$ (harmonic case). Then $$\frac{\mathrm{d}M_c^n}{\mathrm{d}r}(r)\Bigr\rvert_{r=1}=\frac{2^{2-n}}{\sqrt{\pi}}\frac{\Gamma(n/2)}{\Gamma((n-1)/2)}\int_{\alpha(c)}^\pi\frac{\sin^{n-2}t}{\sin^n(t/2)}\mathrm{d}t.$$
\end{lem}
\begin{proof}
    We will use the following notation $T(r)=\frac{1-M_c^n(r)}{1-r}$. Then:
    $$\frac{\mathrm{d}M_c^n}{\mathrm{d}r}(r)\Bigr\rvert_{r=1} = \lim\limits_{r\rightarrow 1^-}T(r).$$
    By using formula (\ref{extremal}) we have
  \begin{align*}
    T(r) & = \frac{1-(2P[\chi_{S(c,\tilde{z})}](x)-1)}{1-r}=\frac{2(1-P[\chi_{S(c,\tilde{z})}](x))}{1-r}.\\
    & \mbox{If we use formula } (\ref{kernel}) \mbox{ we get }\\
    T(r) & = \frac{2P[\mathbf{1}-\chi_{S(c,\tilde{z})}](x)}{1-r}=\frac{2P[\chi_{\mathbf{S}^{n-1}\setminus S(c,\tilde{z})}](x)}{1-r}.\\
    & \mbox{ Now, by using version of Proposition \ref{MMThmSphere} we get the folloing important result: }\\
    T(r) & = 2\sigma_*(n)(1+r)\int_{\alpha(c)}^\pi\frac{\sin^{n-2}t}{(1-2r\cos t+r^2)^{n/2}}\mathrm{d}t
  \end{align*}
We can reformulate this equation, in the following manner
  $$T(r)=2\sigma_*(n)(1+r)\int_{\alpha(c)}^\pi Q(r,t)\mathrm{d}t,$$
  where $Q(r,t)=\frac{\sin^{n-2}t}{(1-2r\cos t+r^2)^{n/2}}$. Since we have limit of proper integral in the last expression, we can derive next formula
  $$\frac{\mathrm{d}M_c^n}{\mathrm{d}r}(r)\Bigr\rvert_{r=1} = 4\sigma_*(n)\int_{\alpha(c)}^\pi \frac{\sin^{n-2}t}{2^n\sin^n(t/2)}\mathrm{d}t.$$
\end{proof}

Let us define $$D_n(a)=\frac{2^{2-n}}{\sqrt{\pi}}\frac{\Gamma(n/2)}{\Gamma((n-1)/2)}\int_{\alpha(c)}^\pi\frac{\sin^{n-2}t}{\sin^n(t/2)}\mathrm{d}t.$$
Then the next theorem holds:
\begin{thm}\label{harmonic_bugheth}
Suppose that $f:\mathbf{B}^n\rightarrow\mathbf{B}^m$ is a harmonic function, such that $f(0)=a_0$, and $f$ has a continuous extension to the point $x_0\in\partial\mathbf{B}^n$ such that $f(x_0)=y_0\in\partial\mathbf{B}^m$.

Then  $\limsup\limits_{r\rightarrow 1_-}  \vert D_rf(rx_0)\vert \geq D_{n}(a)$.

If, in addition, $f$ has a differentiable continuation at point $x_0$, then there exists a positive number $\lambda\in\mathbb{R}$ such that ${Df(x_0)}^{*}y_0=\lambda x_0$ and
 $$\lambda  \geq D_{n}(a),$$
where $a=\langle a_0, y_0\rangle$. This is sharp.

\end{thm}
\begin{proof}
  Let us define function $h(x)=\langle f(x),y_0\rangle$. This function is harmonic (respectively hyperbolic harmonic) in $\mathbf{B}^n$, with $h(0)=a$, and $h(x_0)=1$. Since, by the Theorem of Fotou $M_c^n(1)=1$ we have an implication $$\frac{h(x_0)-h(rx_0)}{1-r}\geq\frac{1-M_c^n(r)}{1-r}.$$

  If $u(r)=h(rx_0),r\in [0,1)$ then $u'(r)=Dh(rx_0)x_0=D_rh(rx_0)$. From Lagrange's theorem we have that for every $r\in[0,1)$ there exists $r_0\in(r,1)$ such that
  $$\frac{1-u(r)}{1-r}=u'(r_0)=D_rh(r_0x_0)\geq\frac{1-M_c^n(r)}{1-r}.$$
  This means that $\limsup\limits_{r\rightarrow 1^-}D_rh(rx_0)\geq\liminf\limits_{r\rightarrow 1_-}\frac{1-u(r)}{1-r}\geq D_{n}(a)$.
  Cauchy -Schwarz inequality provides us that $\vert D_rf(x)\vert \geq D_rh(x)$, which gives us
   $$\limsup\limits_{r\rightarrow 1^-}  \vert D_rf(rx_0)\vert \geq D_{n}(a).$$
\end{proof}

At the end of this section we will investigate whether or not we can formulate the similar version of Schwarz lemma on the boundary, for hyperbolic harmonic functions.

\begin{lem}
Let $(\nu,\mu)=(n-1,n-1)$, where $n>2$ (hyperbolic-harmonic case). Then
$$\frac{\mathrm{d}M_c^n}{\mathrm{d}r}\left(r\right)\Bigr\rvert_{r=1}=0.$$
\end{lem}
\begin{proof}
  Like in previous lemma, we have $$\frac{\mathrm{d}M_c^n}{\mathrm{d}r}(r)\Bigr\rvert_{r=1} = \lim\limits_{r\rightarrow 1^-}T(r),$$

  Let us define $Q_{hyp}(r,t)=\frac{\sin^{n-2}t}{(1-2r\cos t+r^2)^{n-1}}$. Then $$T(r)=2\sigma_*(n)(1-r)^{n-2}(1+r)^{n-1}\int_{\alpha(c)}^\pi Q_{hyp}(r,t)\mathrm{d}t.$$
  Also, define $J_{hyp}(r)=\int_{\alpha(c)}^\pi Q_{hyp}(r,t)\mathrm{d}t$.
  We can pass with the limit, under proper integral sign, to get
$$\lim\limits_{r\rightarrow 1^-}J_{hyp}(r)=J_{hyp}=\int_{\alpha(c)}^\pi q_{hyp}(t) dt,$$   where
$q_{hyp}(t)=4^{-n+1}\sin^{n-2}t \sin^{-2(n-1)}t/2 $.\\
From this we can draw a conclusion $T(r)\sim d_n(1-r)^{n-2}, r\rightarrow 1^-$. This immediately gives our assertion.
\end{proof}
By this Lemma we conclude that we have different situation concerning hyperbolic-harmonic function, mapping unit ball in $\RR^n$ into unit disc in $\RR^m$, in comparison with the harmonic function in same settings. Namely, we found explicit hyperbolic-harmonic function, that maps unit ball into interval $(-1,1)$ such that $u(x_0)=1$, for some $x_0$ on the boundary of the unit ball, but radial derivative in the point $x_0$ is vanishing.

At the first glance, this may look as surprise, having in mind famous Hopf lemma. Function $u$ is satisfying $L(u)=0$, where $L$ is uniformly elliptical partial diferential operator of second order, it has global maximum at point $x_0$ on the boundary of unit ball, so we expected that normal derivative in the point $x_0$ must be grater than zero.

Let $\Omega$ be a domain in $\mathbb{R}^n,n\geq 2$, $x\in\Omega$ be a point and $u$ belongs to $C^2(\Omega)$. We define
$$Lu=a^{ij}(x) D_{ij}u+b^i(x) D_iu+c(x)u, a^{ij}=a^{ji}.$$

The summation convention that repeated indices indicate summation from 1 to n is followed here. We adopt the following definitions: operator $L$ is \textit{elliptic} in point $x\in\Omega$ if the coefficient matrix $A(x)=[a^{ij}(x)]$ is positive definite. If $\Lambda(x),\lambda(x)$ are the greatest and the smallest eigenvalue of matrix $A(x)$ and $\Lambda/\lambda$ is bounded in $\Omega$ we say that $L$ is \textit{uniformly elliptic} in $\Omega$. Also we will need next condition. Let $k>0$ is a constant and $x\in\Omega$ be an arbitrary
\begin{equation}\label{BCondition}
  \frac{\vert b^i(x)\vert }{\lambda(x)}\leq k, i=1,\ldots,n.
\end{equation}
Now, we can formulate Hopf Lemma

\begin{lem}[Hopf lemma,\cite{gilbarg_trudinger}, Lemma 3.4]
  Suppose that $L$ is uniformly elliptic operator, that satisfies condition (\ref{BCondition}), $c=0$ and $Lu\geq 0$ in $\Omega$. Let $x_0\in\partial\Omega$ be such that
  \begin{itemize}
    \item[(i)] $u$ is continuous at $x_0$;
    \item[(ii)] $u(x_0)>u(x)$ for all $x\in\Omega$;
    \item[(iii)] $\partial\Omega$ satisfies an interior sphere condition at $x_0$.
  \end{itemize}
  Then the outer normal derivative of $u$ at $x_0$, if it exists, satisfies the strict inequality
  \begin{equation*}
    \frac{\partial}{\partial\nu}u(x_0)>0.
  \end{equation*}
\end{lem}

What turns out to be is that Hopf lemma demands some conditions on the coefficients standing by the first order derivatives of the elliptic partial differential operator L, that hyperbolic-harmonic functions does not satisfies. We have that hyperbolic-harmonic functions satisfies $L u=\Delta_0 u=0$, where $A(x)=Id$ and $b_i(x)=\frac{2(n-2)}{1-\vert  x\vert  ^2},i=1,\ldots,n$. Since $\Lambda(x)=\lambda(x)=1, x\in\mathbf{B}^n$, we conclude that operator $\Delta_0$ does not satisfies condition (\ref{BCondition}) in $\mathbf{B}^n$, so we can not apply Hopf lemma in this situation.

A part of our result can be interpreted as a confirmation that condition (\ref{BCondition}) can not be excluded from the statement of Hopf lemma.



\section{Appendix}\label{s_App}

Motivated by the role of the Schwarz lemma in Complex Analysis and
numerous fundamental results, see for instance \cite{MM-JMAA,Osserman} and
references therein, in 2016, the first author \cite{RGschMM1} has posted
on ResearchGate the project “Schwarz lemma, the Carathéodory and
Kobayashi Metrics and Applications in Complex
Analysis”.\footnote{Various discussions regarding the subject can also
be found in the Q\&A section on ResearchGate under the question “What
are the most recent versions of the Schwarz lemma?” \cite{RGschMM2}.}

In this project and in \cite{MM-JMAA}, cf. also  \cite{kavu}, we developed
the method related to holomorphic mappings with strip codomain (we refer
to this method as the approach via the Schwarz–Pick lemma for
holomorphic maps from the unit disc into a strip; shortly ”planar strip method”). It is worth mentioning
that the Schwarz lemma has been generalized in various directions; see
\cite{RGschMM1,CHPRbanach}  and the references therein.

Even in planar case researches had some difficulties in handling Schwarz
lemma for harmonic maps of the unit disc into self which does not fix the
origin.
It seems that the researchers have overlooked Burgeth  and H.~W.~Hethcote
results  and they have had  some difficulties to handle the case
$f(0)\neq 0$ in this context; see  for more  details
\cite{MMKhal,MaSvAADM,BiltenAcad}.

In joint paper of the first author with  M. Svetlik  \cite {MaSvAADM}
using   ''planar strip method''  which is  a completely different
approach than B. Burgeth\cite{burgeth1992}, we get a simple proof of an
optimal version of the Schwarz lemma for real valued harmonic functions
(without the assumption that $0$ is mapped to $0$ by the corresponding
map), 
which improves H.~W.~Hethcote result.\footnote{Note here that Burget's
spherical cap method  yield optimal estimate in both planar and spatial
case}.

In joint paper of the first author with  A. Khalfallah and   M. Mhamdi
\cite{KaMaMedJ}, some properties of mappings admitting a Poisson-type
integral representations and the boundary Schwarz lemma were considered.

Presently on this project  the first author  works  with some of his
associates:
A. Khalfallah, M. Arsenovi\'c,  M.Svetlik,  M. Mhamdi, B. Purti\'c, H.P.
Li, J. Gaji\'c  and the second author of this paper.

Chinese mathematicians have made a great contribution to this field but
here we will mention only some whose results are
related to our results. For some interesting complex $n$-dimensional generalisations of classical Schwarz lemma type results see Jian-Feng Zhua's articles \cite{JZhu-RIM} and \cite{JZhu2016}. In paper \cite{TangLiuZhang} the authors proved Schwarz lemma
on the boundary for holomorphic mappings between
unit balls in $\mathbb{C}^n$, and some of theirs rigidity properties.
Generalization of this theorem, for separable complex Hilbert space was
given by Z. Chen, Y. Liu and Y. Pan in \cite{ChenLiuPan}. While proving Proposition 2.9,
we independently proved Lemma 2.7, but later found
that result proven in \cite{ChenLiuPan}, as it can be seen in corresponding reference.
In \cite{DaiChenPan} the authors proved a higher order Schwarz-Pick lemma for holomorphic
mappings between unit balls in complex Hilbert
spaces.

For generalizations  of  Schwarz lemmas for  planar harmonic mappings into
the
sharp forms of Banach spaces we refer  the  interested reader to  Chen,
Hamada et al. \cite{CHfunct,CHPRbanach} and literature cited there for the
background.
Recall  the main purpose of the paper \cite{CHPRbanach} is to develop some
methods to investigate
the Schwarz type lemmas for holomorphic mappings and pluriharmonic
mappings in Banach spaces. Initially, they  extend the classical Schwarz
lemmas
for holomorphic mappings to Banach spaces.
Furthermore, they improve and
generalize the classical Schwarz lemmas for planar harmonic mappings and
obtain sharp versions for Banach spaces, and present some applications to sharp boundary
Schwarz type lemmas for pluriharmonic mappings in Banach spaces. The
obtained
results provide improvements and generalizations of the corresponding
results
in \cite{CHfunct} (cf. also\cite{Li_MM}).
\vskip.2cm

\textbf{Acknowledgments} The authors are indebted to M. Arsenovi\'c for an interesting discussions on this paper.

\nocite{*}
\bibliography{sn-bibliography}


\end{document}